\documentclass[11pt]{article}
\usepackage{fullpage,amsthm,amssymb,amsmath, amsfonts}
\usepackage{graphicx,algorithmic,algorithm}
\usepackage{enumerate}
\usepackage{tikz}
\usetikzlibrary{shapes}
\usepackage{hyperref}
\usepackage[normalem]{ulem}
\hypersetup{
    pdftitle=   {Coloring graphs without fan vertex-minors and graphs without cycle pivot-minors},
   pdfauthor=  {Ilkyoo Choi, O-joung Kwon, and Sang-il Oum}
}
\usepackage{authblk}
\newtheorem{theorem}{Theorem}[section]
\newtheorem{lemma}[theorem]{Lemma}
\newtheorem{proposition}[theorem]{Proposition}

\newtheorem{conjecture}[theorem]{Conjecture}

\newtheorem*{THMMAIN}{Theorem \ref{thm:mainfanvertexminor}}
\newtheorem*{THMMAIN2}{Theorem \ref{thm:mainpivotminor}}

\usetikzlibrary{patterns}
\usetikzlibrary{fadings}
\usetikzlibrary{shapes,snakes}
\usetikzlibrary{decorations.markings}
\usetikzlibrary{decorations.pathreplacing}

\newcommand\abs[1]{\lvert #1\rvert}

\newcommand{\cC}{\mathcal{C}}

\begin{document}
\title{Coloring graphs without fan vertex-minors and graphs without cycle pivot-minors}

\author[1]{Ilkyoo Choi\thanks{Supported by the National Research Foundation of Korea (NRF) grant funded by the Korea government (MSIP) (NRF-2015R1C1A1A02036398).}}
\affil[1]{Department of Mathematical Sciences, KAIST,  Daejeon, South Korea.}

\author[2]{O-joung Kwon\thanks{Supported by ERC Starting Grant PARAMTIGHT (No. 280152).}}

\affil[2]{Institute for Computer Science and Control, Hungarian Academy of Sciences, Budapest, Hungary.}

\author[1]{Sang-il Oum\thanks{Supported by Basic Science Research
  Program through the National Research Foundation of Korea (NRF)
  funded by  the Ministry of Science, ICT \& Future Planning
  (2011-0011653).}}

\date\today
\maketitle

\footnotetext{E-mail addresses: \texttt{ilkyoo@kaist.ac.kr} (I. Choi), \texttt{ojoungkwon@gmail.com} (O. Kwon), \texttt{sangil@kaist.edu} (S. Oum)}

\begin{abstract}
A fan $F_k$ is a graph that consists of an induced path on $k$ vertices and an additional vertex that is adjacent to all vertices of the path.
We prove that for all positive integers $q$ and $k$, every graph with sufficiently large chromatic number contains either a clique of size $q$ or a vertex-minor isomorphic to $F_k$. 
We also prove that for all positive integers $q$ and $k\ge 3$, every graph with sufficiently large chromatic number contains either a clique of size $q$ or a pivot-minor isomorphic to a cycle of length $k$.
\end{abstract}

\section{Introduction}
All graphs in this paper are simple, which means no loops and no parallel edges.
Given a graph, a \emph{clique} is a set of pairwise adjacent vertices and an \emph{independent set} is a set of pairwise non-adjacent vertices. 
For a graph $G$, let $\chi (G)$ denote the \emph{chromatic number} of $G$ and let $\omega (G)$ denote the maximum size of a clique of $G$. 
Since two vertices in a clique cannot receive the same color in a proper coloring, the clique number is a trivial lower bound for the chromatic number. 
If $\chi(H)=\omega(H)$ for every induced subgraph $H$ of a graph $G$, then we say $G$ is \emph{perfect}.
Gy\'arf\'as~\cite{Gyarfas1987} introduced the notion of a $\chi$-bounded class as a generalization of perfect graphs.
A class $\cC$ of graphs is \emph{$\chi$-bounded} if there exists a function $f: \mathbb{N} \rightarrow \mathbb{N}$ 
such that for all graphs $G\in \cC$, and all induced subgraphs $H$ of $G$, $\chi (H) \le  f(\omega (H))$.
Therefore the class of perfect graphs is $\chi$-bounded with the identity function.

Chudnovsky, Robertson, Seymour, and Thomas~\cite{ChudnovskyRST2006} proved the strong perfect graph theorem,
 which states that a graph $G$ is perfect if and only if neither $G$ nor its complement contains an induced odd cycle of length at least $5$.  
This shows that there is a deep connection between the chromatic number and the structure of the graph. 
Gy\'arf\'as~\cite{Gyarfas1987} proved that for each integer $k$, the class of graphs with no induced path of length $k$ is $\chi$-bounded.
Gy\'arf\'as also made the following three conjectures for $\chi$-boundedness in terms of forbidden induced subgraphs.
Note that (iii) implies both (i) and (ii).

\begin{conjecture}[Gy\'arf\'as~\cite{Gyarfas1987}]\label{con:gy}
The following classes are $\chi$-bounded:
\begin{enumerate}[\normalfont (i)]
\item The class of graphs with no induced odd cycle of length at least $5$.
\item The class of graphs with no induced cycle of length at least $k$
  for a fixed $k$.
\item The class of graphs with no induced odd cycle of length at least
  $k$ for a fixed $k$.
\end{enumerate}
\end{conjecture}

There are recent works by Chudnovsky, Scott, and Seymour~\cite{ChScSe_2,ChScSe_3,ChScSe_5} and Scott and Seymour~\cite{ScSe_1,ScSe_4} regarding $\chi$-boundedness and induced subgraphs; in this series of papers they prove (i) and (ii) of Conjecture~\ref{con:gy},  and also solve the case when $k=5$ for (iii).
The full conjecture of (iii) is still open. 
One result in this paper (Theorem~\ref{thm:mainpivotminor}) gives further evidence on  (iii) of Conjecture~\ref{con:gy}, as the half of Theorem~\ref{thm:mainpivotminor} is implied by (iii) of Conjecture~\ref{con:gy}.

Scott and Seymour~\cite{ScSe_4} proved that the class of triangle-free graphs having no long induced even (or odd) cycles have bounded chromatic number, thus extending the result of Lagoutte~\cite{Aurelie2015} who claimed a proof for triangle-free graphs having no induced even cycles of length at least $6$.
It has also been shown that the class of graphs having no induced even cycle~\cite{AddarioCHRS2008} is $\chi$-bounded.

The following graph classes are also known to be $\chi$-bounded:
\begin{itemize}
\item Bipartite graphs, distance-hereditary graphs, and parity graphs are perfect graphs and therefore $\chi$-bounded~\cite{BandeltM86, Butlet88}.
\item Circle graphs are $\chi$-bounded, shown by Kostochka and Kratochv{\'{\i}}l~\cite{KostochkaK1997}.
\item For each integer $k$, the class of graphs of rank-width at most $k$ is $\chi$-bounded, shown by
Dvo\v{r}\'{a}k and Kr\'{a}l'~\cite{DvorakK2012}.
\end{itemize}

\emph{Vertex-minors} and \emph{pivot-minors} are graph containment relations introduced by Bouchet~\cite{Bouchet1987a, Bouchet1987b, Bouchet1988, Bouchet1989a} while conducting research of circle graphs (intersection graphs of chords in a cycle) and $4$-regular Eulerian digraphs.
Furthermore, these graph operations have been used for developing theory on rank-width~\cite{JKO2014,Oum2004,Oum2004a, Oum2006a, Oum12}.
We review these concepts in Section~\ref{sec:prelim}. 
Interestingly, the aforementioned graph classes can be characterized
in terms of forbidden vertex-minors or pivot-minors.
\begin{itemize}
\item Bipartite graphs are graphs having no pivot-minor isomorphic to $C_3$. 
\item Parity graphs are graphs having no pivot-minor isomorphic to $C_5$
\footnote{Parity graphs are known as graphs admitting a split decomposition whose bags are bipartite graphs or complete graphs~\cite{CiceroneS1999}, and it implies that parity graphs are closed under taking pivot-minors. One can easily verify that parity graphs are $C_5$-pivot-minor-free graphs using the fact that parity graphs are the graphs in which every odd cycle has two crossing chords~\cite{Butlet88}.}.  
\item Distance-hereditary graphs are graphs having no vertex-minor isomorphic to $C_5$, shown by Bouchet~\cite{Bouchet1987a,Bouchet1988}.
\item Circle graphs are graphs having no vertex-minor isomorphic to the three graphs in Figure~\ref{fig:vmobscircle}, shown by Bouchet~\cite{Bouchet1994}. 
Circle graphs are graphs having no pivot-minor isomorphic to the fifteen graphs, shown by Geelen and Oum~\cite{GO2009}.
\item Graphs of rank-width at most $k$ can be characterized by a finite list of forbidden pivot-minors, shown by Oum~\cite{Oum2004, Oum2004a}.
\end{itemize} 

	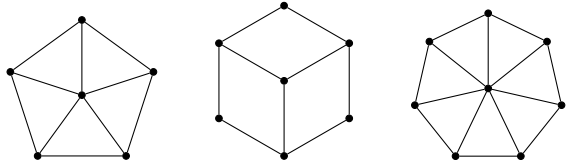
\begin{figure}\centering
\tikzstyle{v}=[circle, draw, solid, fill=black, inner sep=0pt, minimum width=2.5pt]
\tikzset{photon/.style={decorate, decoration={snake}}}
  \begin{tikzpicture}
    \node[v](v0) at (0,0){};
    \foreach \x in {1,2,3,4,5}
    {
      \node[v](v\x) at (\x*72+90:1) {};
      \draw (v0)--(v\x);
    }
      \draw (v1)--(v2)--(v3)--(v4)--(v5)--(v1);
    \end{tikzpicture}\qquad
  \begin{tikzpicture}
    \node[v](v0) at (0,0){};
    \foreach \x in {1,3,5}
    {
      \node[v](v\x) at (\x*60+90:1) {};
      \draw (v0)--(v\x);
    }
    \foreach \x in {2,4,6}
    {
      \node[v](v\x) at (\x*60+90:1) {};
    }
    \draw (v1)--(v2)--(v3)--(v4)--(v5)--(v6)--(v1);
   \end{tikzpicture}\qquad
  \begin{tikzpicture}
    \node[v](v0) at (0,0){};
    \foreach \x in {1,2,3,4,5,6,7}
    {
      \node[v](v\x) at (\x*360/7+90:1) {};
      \draw (v0)--(v\x);
    }
      \draw (v1)--(v2)--(v3)--(v4)--(v5)--(v6)--(v7)--(v1);
    \end{tikzpicture}
\caption{The three forbidden vertex-minors for circle graphs.}
\label{fig:vmobscircle}
\end{figure}

In 2009, Geelen (see~\cite{DvorakK2012}) conjectured the following, which includes all aforementioned results regarding classes of graphs excluding certain vertex-minors.
\begin{conjecture}[Geelen]\label{con:geelen}
For every graph $H$, the class of graphs having no vertex-minor isomorphic to $H$ is $\chi$-bounded.
\end{conjecture}
Dvo\v{r}\'{a}k and Kr\'{a}l'~\cite{DvorakK2012} showed that Conjecture~\ref{con:geelen} is true when $H=W_5$, where $W_5$ is the wheel graph on $6$ vertices, depicted in Figure~\ref{fig:vmobscircle}.
Chudnovsky, Scott, and Seymour~\cite{ChScSe_3}  showed that 
(ii) of Conjecture~\ref{con:gy} holds and this implies that Conjecture~\ref{con:geelen} is true when $H$ is a cycle.

In 1997, Scott~\cite{Scott1997} made a stronger conjecture claiming that for every graph $H$,  the class of graphs having no subdivision of $H$ as an induced subgraph
is $\chi$-bounded and proved the conjecture when $H$ is a tree as follows.
However, the conjecture of Scott turned out to be false, shown by Pawlik et al.~\cite{PKKLMTW2014}. 
\begin{theorem}[Scott~\cite{Scott1997}]\label{thm:scott}
  For every tree $H$, the class of graphs having no induced subdivision of $H$ is $\chi$-bounded.
\end{theorem}

Theorem~\ref{thm:scott} implies that Conjecture~\ref{con:geelen} is true when $H$ is a vertex-minor of a tree.
Kwon and Oum~\cite{KO2013} showed that a graph is a vertex-minor of some tree if and only if it is a distance-hereditary graph, or equivalently, a graph of rank-width $1$.  Thus, Theorem~\ref{thm:scott} implies that Conjecture~\ref{con:geelen} is true if $H$ is a distance-hereditary graph. 

\bigskip

Our main theorem (Theorem~\ref{thm:mainfanvertexminor}) adds another infinite class of graphs for which Conjecture~\ref{con:geelen} is true.
A \emph{fan} $F_k$ is a graph that consists of an induced path on $k$ vertices and an additional vertex not on the path that is adjacent to all vertices of the path.
We prove the following.

\begin{THMMAIN}
  For each integer $k$, the class of graphs having no vertex-minor isomorphic to $F_k$ is $\chi$-bounded.
\end{THMMAIN}

We further ask whether the stronger statement for pivot-minors is also true. 
Conjecture~\ref{con:geelen} would be true if Conjecture~\ref{con:pivotminor} were to be true, because every pivot-minor of a graph is a vertex-minor.

\begin{conjecture}\label{con:pivotminor}
For every graph $H$, the class of graphs having no pivot-minor isomorphic to $H$ is $\chi$-bounded.
\end{conjecture}

Theorem~\ref{thm:scott} implies that if $H$ is a subdivision of $K_{1,n}$, then Conjecture~\ref{con:pivotminor} is true. Thus, Conjecture~\ref{con:pivotminor} is true when $H$ is a pivot-minor of a subdivision of $K_{1,n}$.

Scott and Seymour~\cite{ScSe_1} proved that the class of graphs with no odd hole is $\chi$-bounded, proving (i) of Conjecture~\ref{con:gy}. Thus, Conjecture~\ref{con:pivotminor} holds when $H=C_5$.
Our second theorem provides another evidence to Conjecture~\ref{con:pivotminor} as follows.

\begin{THMMAIN2}
For each integer $k\ge 3$, the class of graphs having no pivot-minor  isomorphic to a cycle of length $k$ is $\chi$-bounded.
\end{THMMAIN2}

Theorem~\ref{thm:mainpivotminor} does not follow from the result of Chudnovsky, Scott, and Seymour~\cite{ChScSe_3} on long holes.
The reason is that for every pair of integers $k$ and $\ell$ with $k>\ell$ and $k-\ell\equiv 1\pmod 2$, $C_k$ has no pivot-minor isomorphic to $C_{\ell}$ 
\footnote{This can be checked using the result of Bouchet~\cite{Bouchet1988} that if $H$ is a pivot-minor of $G$ and $v\in V(G)\setminus V(H)$, then 
$H$ is a pivot-minor of one of $G\setminus v$ and $G\wedge vw\setminus v$ for a neighbor $w$ of $v$. It implies that if $C_{\ell}$ is isomorphic to a pivot-minor of $C_{k}$ and $k>\ell$, then 
$C_{\ell}$ is isomorphic to a pivot-minor of $C_{k-2}$.}
(but has a pivot-minor isomorphic to every shorter induced cycle with the same parity).
We would like to mention that if (iii) of Conjecture~\ref{con:gy} were to be true, 
then this would imply Conjecture~\ref{con:pivotminor} is true when $H$ is an odd cycle.

The paper is organized as follows. In Section~\ref{sec:prelim}, we provide necessary definitions including vertex-minors, pivot-minors, and a leveling of a graph.
Section~\ref{sec:3} proves Theorem~\ref{thm:mainfanvertexminor}. We show that for a leveling of a graph, if a level contains a sufficiently long induced path, then the graph contains a large fan as a vertex-minor. 
We devote in Subsections~\ref{sec:pathwithmatching} and \ref{sec:patchedpath} to show how to find a simple structure containing a fan vertex-minor from a leveling with a long induced path in a level. 
With the help of a result by Gy\'{a}rf\'{a}s~\cite{Gyarfas1987} (Theorem~\ref{thm:gyarfas}) 
we show Theorem~\ref{thm:mainfanvertexminor} in Subsection~\ref{sec:fanfreegraphs}.
Section~\ref{sec:4} presents a proof of Theorem \ref{thm:mainpivotminor} by using a similar strategy. 
However, there is an issue of finding a pivot-minor isomorphic to a long induced cycle from a graph consisting of a long induced path with a vertex having many neighbors on it. 
In fact, this is not always true; for instance, a graph obtained from a fan by subdividing each edge on the path once is bipartite, and thus, it contains no odd cycles. 
We need a relevant result regarding the parity of a cycle, and we show
in Subsection~\ref{sec:findingfan} that for every fixed $k$, there exists
$\ell$ with $\ell\equiv k\pmod 2$ such that every graph consisting of an induced path $P$ of length $\ell$ and a vertex $v$ not on $P$ where  $v$ is adjacent to the end vertices of $P$ and may be adjacent to some other vertices contains a pivot-minor isomorphic to $C_k$.
Based on this result, we show Theorem~\ref{thm:mainpivotminor} in Subsection~\ref{sec:nockpmchibound}.
We conclude the paper by further discussions in Section~\ref{sec:discuss}.

\section{Preliminaries}\label{sec:prelim} 
For a graph $G$, let $V(G)$ and $E(G)$ denote the vertex set and the edge set of $G$, respectively. 
For $S\subseteq V(G)$, let $G[S]$ denote the subgraph of $G$ induced on the vertex set $S$. For $v\in V(G)$ and $S\subseteq V(G)$, let $G\setminus v$ be the graph obtained from $G$ by removing $v$, and let $G\setminus S$ be the graph obtained by removing all vertices in $S$. 
For $F\subseteq E(G)$, let $G\setminus F$ denote the graph obtained from $G$ by removing all edges in $F$.
For $v\in V(G)$, the set of \emph{neighbors} of $v$ in $G$ is denoted by $N_G(v)$.

The \emph{length} of a path is the number of edges on the path.

For two positive integers $k$ and $\ell$, let $R(k,\ell)$ be the \emph{Ramsey number}, which is the minimum integer satisfying that every graph with at least $R(k,\ell)$ vertices contains either a clique of size $k$ or an independent set of size $\ell$. By Ramsey's Theorem~\cite{Ramsey1930}, $R(k, \ell)$ exists for every pair of positive integers $k$ and $\ell$.

\subsection*{Vertex-minors and pivot-minors}

Given a graph $G$ and a vertex $v\in V(G)$, let $G*v$ denote the graph obtained from $G$ by applying local complementation at $v$; the \emph{local complementation} at $v$ is an operation to replace the subgraph induced on $N_G(v)$ with its complement.
A graph $H$ is a \emph{vertex-minor} of $G$ if $H$ can be obtained from $G$ by applying a sequence of local complementations and vertex deletions.

The graph obtained from $G$ by \emph{pivoting} an edge $uv\in E(G)$ is defined by $G\wedge uv:=G*u*v*u$. 
A graph $H$ is a \emph{pivot-minor} of $G$ if $H$ can be obtained from $G$ by pivoting edges and deleting vertices.
By the definition of pivoting edges, every pivot-minor of a graph $G$ is also its vertex-minor.

For an edge $uv$ of a graph $G$, let $S_1:= N_G(u)\setminus (N_G(v)\cup \{v\})$, $S_2:=N_G(v)\setminus (N_G(u)\cup \{u\})$, and $S_3:=N_G(v)\cap N_G(u)$. 
See Figure~\ref{fig:pivoting} for an example. 
It is easy to verify that $G\wedge uv$ is identical to the graph obtained from $G$ by complementing the adjacency relations of vertices between distinct sets $S_i$ and $S_j$, and swapping the labels of the vertices $u$ and $v$. 
See \cite[Proposition 2.1]{Oum2004} for a formalized proof.

\begin{figure}\centering
\tikzstyle{v}=[circle, draw, solid, fill=black, inner sep=0pt, minimum width=2.5pt]
\tikzset{photon/.style={decorate, decoration={snake}}}
\begin{tikzpicture}[scale=0.05]

\node [v] (a) at (10,50) {};
\node [v] (b) at (30,50) {};
\node [v] (c) at (0,30) {};
\node [v] (d) at (10,15) {};
\node [v] (e) at (30,17) {};
\node [v] (f) at (40,35) {};
\node [v] (g) at (40,10) {};

\draw [rounded corners=5pt, thick] (-3,33) rectangle (13, 12);
\draw (5, 7) node{$S_1$};
\draw [rounded corners=3pt, thick] (27,20) rectangle (33, 14);
\draw (30, 9) node{$S_2$};
\draw [rounded corners=3pt, thick] (37,38) rectangle (43, 32);
\draw (47, 35) node{$S_3$};

\draw (c) -- (a) -- (b) -- (f);
\draw (d) -- (a);
\draw (e) -- (b);
\draw (f) -- (c) -- (d);
\draw (c)-- (e);
\draw (a)--(f);

\draw (d) -- (g) -- (f);

\draw (10,55) node{$u$};
\draw (30,55) node{$v$};

\end{tikzpicture}\qquad\quad
\begin{tikzpicture}[scale=0.05]

\node [v] (a) at (10,50) {};
\node [v] (b) at (30,50) {};
\node [v] (c) at (0,30) {};
\node [v] (d) at (10,15) {};
\node [v] (e) at (30,17) {};
\node [v] (f) at (40,35) {};
\node [v] (g) at (40,10) {};

\draw [rounded corners=5pt, thick] (-3,33) rectangle (13, 12);
\draw (5, 7) node{$S_1$};
\draw [rounded corners=3pt, thick] (27,20) rectangle (33, 14);
\draw (30, 9) node{$S_2$};
\draw [rounded corners=3pt, thick] (37,38) rectangle (43, 32);
\draw (47, 35) node{$S_3$};

\draw (c) -- (a) -- (b) -- (f);
\draw (d) -- (a);
\draw (e) -- (b);
\draw (c) -- (d);
\draw (a)--(f);

\draw (d) -- (e) --(f);
\draw (d) -- (f);

\draw (d) -- (g) -- (f);

\draw (10,55) node{$v$};
\draw (30,55) node{$u$};

\end{tikzpicture}\caption{Pivoting an edge $uv$.}\label{fig:pivoting}
\end{figure}
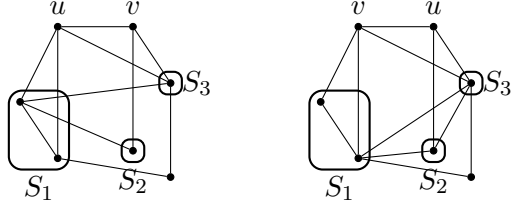

For a vertex $v$ of $G$ with exactly two neighbors $v_1$ and $v_2$, if $v_1$ and $v_2$ are non-adjacent, then the operation of replacing $G$ with $G*v\setminus v$ is called \emph{smoothing} a vertex $v$.
Smoothing a vertex $v$ is equivalent to removing $v$ and adding the edge between the two neighbors of $v$.

\subsection*{Leveling in a graph}

A sequence $L_0, L_1, \ldots, L_m$ of disjoint subsets of the vertex set of a graph $G$ is called a \emph{leveling} in $G$ if 
\begin{enumerate}
\item $\abs{L_0} = 1$, and
\item for each $i\in\{1, \ldots, m\}$, every vertex in $L_i$ has a neighbor in $L_{i-1}$, and has no neighbors in $L_j$ for all $j\in\{0, \ldots, i-2\}$.
\end{enumerate}
Each $L_i$ is called a \emph{level}. 
For $i\in\{1, \ldots, m\}$, a vertex $v\in L_{i-1}$ is called a \emph{parent} of a vertex $w\in L_i$ if $v$ and $w$ are adjacent in $G$.
For $u\in L_i$ and $v\in L_j$ where $0\le i \le j \le m$, 
$u$ is called an \emph{ancestor} of $v$ 
if there is a path between $u$ and $v$ of length $j-i$ with one vertex in each of $L_i, L_{i+1}, \ldots ,L_j$.

One natural way to obtain a leveling that covers all vertices in a graph is to fix a vertex $v$, and define $L_i$ as the set of all vertices at distance $i$ from $v$.

Our basic strategy to color a graph is to color each level of this leveling.
If each level can be colored with $N$ colors, then all levels can be colored with $2N$ colors, by using two disjoint sets of $N$ colors for even levels and odd levels. 
So, we may assume that some level has sufficiently large chromatic number.
The following theorem of Gy\'{a}rf\'{a}s~\cite{Gyarfas1987} implies that we may assume that a level contains a sufficiently long induced path, and this gives a starting point of proving Theorems~\ref{thm:mainfanvertexminor} and \ref{thm:mainpivotminor}.

\begin{theorem}[Gy\'{a}rf\'{a}s~\cite{Gyarfas1987}]\label{thm:gyarfas}
  If $k\ge 2$ and a graph $G$ has no induced path on $k$ vertices, then $\chi(G)\le(k-1)^{\omega(G)-1}$.
\end{theorem}

\section{Coloring graphs without $F_k$ vertex-minors}
\label{sec:3}
We prove that every class of graphs excluding a fixed fan as a vertex-minor is $\chi$-bounded.

\begin{theorem}\label{thm:mainfanvertexminor}
 For each integer $k$, the class of graphs having no vertex-minor isomorphic to $F_k$ is $\chi$-bounded.
\end{theorem}

\subsection{A structure containing a fan vertex-minor}\label{sec:pathwithmatching}

To show Theorem~\ref{thm:mainfanvertexminor}, we essentially prove that for a fixed $k$ and a graph $G$ with a leveling, 
if a level contains a sufficiently long induced path, then $G$ contains a vertex-minor isomorphic to $F_k$.
In this subsection, we introduce an intermediate structure having a vertex-minor isomorphic to $F_k$.

We will use the following two theorems.

\begin{theorem}[Erd\H{o}s and Szekeres~\cite{1935ErSz}]\label{thm:Erdoss1935}
Every sequence of $n^2+1$ integers contains an increasing or decreasing subsequence of length $n+1$.
\end{theorem}

\begin{theorem}[folklore; see Diestel~{\cite{Diestel2010}}]\label{thm:binary}
  For $k\ge 1$ and $\ell\ge 3$,  
  every connected graph on at least $k^{\ell-2}+1$ vertices contains 
  a vertex of degree at least $k$ or an induced path on $\ell$ vertices.
\end{theorem}

For $k\geq 2$, let $E_k$ be a graph on $3k$ vertices constructed in the following way: start with the disjoint union of $k$ $2$-edge paths $P_1, \ldots, P_k$ having $v_1,\ldots,v_k$ as an end vertex, respectively and  then add $k-1$ edges that make the graph induced on $\{v_1, \ldots, v_k\}$ a path (of length $k-1$). 
Note that $E_k$ is a tree with $k$ vertices of degree $1$, $k+2$ vertices of degree $2$, and $k-2$ vertices of degree $3$.

\begin{proposition}\label{prop:ektofan}
Let $k$ be a positive integer and let $\ell\ge R(k,k)^{2(k-1)^2-1}+1$. Let $H$ be a connected graph with at least $\ell$ vertices. 
Then the graph obtained from the disjoint union of $H$ and $E_\ell$ by identifying $\ell$ distinct vertices of $H$ with the leaves of $E_\ell$ contains a vertex-minor isomorphic to $F_k$.
\end{proposition}
 See Figure~\ref{fig:examplecomb} for an illustration of a graph described in Proposition~\ref{prop:ektofan}.

\begin{figure}
  \centering
  \tikzstyle{v}=[circle, draw, solid, fill=black, inner sep=0pt, minimum width=3pt]
      \begin{tikzpicture}[scale=0.8]
    \foreach \x in {1,...,6}
    {
    }
        \foreach \x in {0,...,2}
    {
      \node (p1\x) at (1,\x*1.2) [v] {};
    }
        \foreach \x in {0,...,2}
    {
      \node (p2\x) at (2,\x*1.2) [v] {};
    }
        \foreach \x in {0,...,2}
    {
      \node (p3\x) at (3,\x*1.2) [v] {};
    }
        \foreach \x in {0,...,2}
    {
      \node (p4\x) at (4,\x*1.2) [v] {};
    }
        \foreach \x in {0,...,2}
    {
      \node (p5\x) at (5,\x*1.2) [v] {};
    }
        \foreach \x in {0,...,2}
    {
      \node (p6\x) at (6,\x*1.2) [v] {};
    }

    \draw[very thick] (p10)--(p12);
    \draw[very thick] (p20)--(p22);
    \draw[very thick] (p30)--(p32);
    \draw[very thick] (p40)--(p42);
    \draw[very thick] (p50)--(p52);
    \draw[very thick] (p60)--(p62);
  \draw[very thick] (p10)--(p60);
 	\draw [rounded corners=2pt, very thick] (1-.2,2.2) rectangle (6+.2, 4.5);

  \end{tikzpicture}
  \caption{A graph obtained from $E_{6}$ and a connected graph by identifying $6$ vertices.}
  \label{fig:examplecomb}
\end{figure}
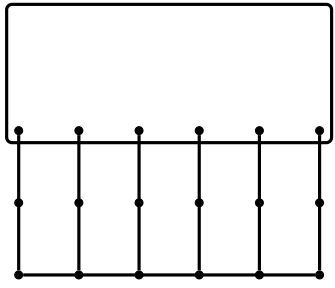

We first observe that for every connected graph $H$ and a vertex $v$ in $H$, either $H\setminus v$ or $H*v\setminus v$ is connected.
This allows us to reduce $H$ into a graph on exactly $\ell$ vertices.

\begin{lemma}\label{lem:connected}
Let $H$ be a connected graph with at least $2$ vertices. 
For each vertex $v$ of $H$, either $H\setminus v$ or $H*v\setminus v$ is connected. 
\end{lemma}
\begin{proof}
If $H[N_H(v)]$ is connected, then $H\setminus v$ is connected trivially. 
Otherwise, $(H*v)[N_H(v)]$ is connected and therefore 
$H*v\setminus v$ is connected.
\end{proof}

This implies that in Proposition~\ref{prop:ektofan}, if $H$ contains a vertex $v$ that will not be identified with a leaf of $E_\ell$, then we can reduce $H$ into one of $H\setminus v$ or $H*v\setminus v$, which is connected. 
In the end, we may assume that $H$ is a connected graph on the vertex set $\{v_1, \ldots, v_{\ell}\}$.
We now aim to obtain a fan vertex-minor in either case, by using Theorem~\ref{thm:binary}, which says that every sufficiently large connected graph contains a vertex of large degree or a long induced path, 

The following lemma proves the case when $H$ contains a long induced path.
For a positive integer $t$, the \emph{ladder} of order $t$ is a graph $G$ that consists of two vertex-disjoint paths $P = p_1p_2 \cdots p_t$, $Q=q_1q_2 \cdots q_t$ such that
\begin{itemize}
\item $V(G)=V(P) \cup V(Q)$, and
\item for each $i, j\in\{1, \ldots, t\}$, $p_iq_j\in E(G)$ if and only if $i=j$.
\end{itemize}
The \emph{$1$-subdivision} of a graph $G$ is the graph obtained from $G$ by replacing each edge by a $2$-edge path.

\begin{lemma}\label{lem:laddertofan}
The $1$-subdivision of the ladder of order $k$ contains a vertex-minor isomorphic to $F_k$.
\end{lemma}
\begin{proof}
Let $H$ be the ladder of order $k$  with two vertex-disjoint paths $P = p_1p_2 \cdots p_k$ and $Q=q_1q_2 \cdots q_k$ such that 
for each $i, j\in\{1, \ldots, k\}$, $p_iq_j\in E(G)$ if and only if $i=j$.
Let $G$ be the $1$-subdivision of $H$, and let $v_{xy}$ be the degree-$2$ vertex adjacent to $x$ and $y$ in $G$ for each edge $xy$ of $H$.
We claim that for each $1\le j\le k-1$, the vertex $p_{j+1}$ is adjacent to $v_{p_iq_i}$ for all $1\le i\le j+1$ in the graph 
\[G\wedge p_1v_{p_1p_2} \wedge \cdots \wedge p_jv_{p_jp_{j+1}}.\]
It is easy to observe that this is true when $j=1$. 
Suppose $j\ge 2$. 
By the induction hypothesis, $p_{j}$ is adjacent to $v_{p_iq_i}$ for all $1\le i\le j$ in the graph 
$G\wedge p_1v_{p_1p_2} \wedge \cdots \wedge p_{j-1}v_{p_{j-1}p_{j}}$.
Note that $v_{p_jp_{j+1}}$ still has two neighbors $p_j$ and $p_{j+1}$ in the graph $G\wedge p_1v_{p_1p_2} \wedge \cdots \wedge p_{j-1}v_{p_{j-1}p_{j}}$
because it is adjacent to no vertex of $\{p_1, v_{p_1p_2}, \ldots, p_{j-1}, v_{p_{j-1}p_j}\}$ in $G$ and thus, it was not affected by the previous pivotings.
By the definition of pivoting, $p_{j+1}$ becomes adjacent to $v_{p_iq_i}$ for all $1\le i\le j+1$ in 
\[(G\wedge p_1v_{p_1p_2} \wedge \cdots \wedge p_{j-1}v_{p_{j-1}p_{j}})\wedge p_jv_{p_jp_{j+1}}.\]

By the above claim, $p_k$ is adjacent to $v_{p_iq_i}$ for all $1\le i\le k$ in 
$G\wedge p_1v_{p_1p_2} \wedge \cdots \wedge p_{k-1}v_{p_{k-1}p_k}$.
Note that there are no edges between the vertices of $\{v_{p_iq_i}:1\le i\le k\}$ as this graph is bipartite.
Therefore, by removing all vertices in $\{p_1, v_{p_1p_2}, \ldots, p_{k-1}, v_{p_{k-1}p_k}\}$
and smoothing all degree-$2$ vertices in the remaining graph, 
we obtain a vertex-minor isomorphic to $F_k$.
\end{proof}

\begin{proof}[Proof of Proposition~\ref{prop:ektofan}]
Let $w_1,\ldots,w_\ell$ be the leaves of $E_\ell$ in the order following the main path. 
For all $i\in\{1, \ldots, \ell\}$, let $x_i$ be the neighbor of $w_i$ in $E_{\ell}$ and let $y_i$ be the neighbor of $x_i$ other than $w_i$.
Let $v_1,\ldots,v_\ell$ be the vertices of $H$ to be identified with $w_1,\ldots,w_\ell$, respectively.
Let $G$ be the graph obtained from the disjoint union of $H$ and $E_\ell$ by identifying $v_i$ and $w_i$ for each $i$.

Suppose there is a vertex $v$ in $H$ other than $v_1, \ldots, v_{\ell}$. 
By Lemma~\ref{lem:connected}, either $H\setminus v$ or $H*v\setminus v$ is connected.
Since applying local complementation at $v$ in $G$ does not change adjacency with a vertex in $V(E_{\ell})\setminus \{w_1, \ldots, w_{\ell}\}$, 
we can reduce $G$ to one of $G\setminus v$ or $G*v\setminus v$.  
By this observation, we may assume that $H$ is a connected graph on the vertex set $\{v_1, \ldots, v_{\ell}\}$.

Since $\ell\ge R(k,k)^{2(k-1)^2-1}+1$, by Theorem~\ref{thm:binary}, $H$ contains a vertex of degree at least $R(k,k)$, or an induced path on $2(k-1)^2+1$ vertices.

\medskip
\textbf{Case 1}: $H$ has an induced path $v_{i_1}v_{i_2} \ldots v_{i_{2(k-1)^2+1}}$. 

By Theorem~\ref{thm:Erdoss1935}, $i_1, i_2,\ldots, i_{2(k-1)^2+1}$ contains an increasing or decreasing subsequence $j_1, j_2,\ldots, j_k$, where
all of $j_1, \ldots, j_k$ have the same parity.
We may assume $j_1<j_2<\cdots<j_k$ by relabeling the indices if necessary and let $j_1=i_p$ and $j_k=i_q$. 
Now, the graph induced on 
\[\{w_z: z\in\{i_p, i_{p+1},\ldots, i_q\}\}\cup \{x_{z}: z\in \{j_1, j_2, \ldots, j_k\}\}\cup \{y_{z}: z\in\{j_1, j_1+1,\ldots, j_k\}\}\] is a subdivision of a ladder of order $k$, where each edge of the ladder is subdivided at least once. 
We apply local complementations to degree-$2$ vertices to transform this graph into the $1$-subdivision of the ladder of order $k$.
By Lemma~\ref{lem:laddertofan}, it contains a vertex-minor isomorphic to $F_k$.

\medskip
\textbf{Case 2}: $H$ has a vertex $v_s$ of degree at least $R(k,k)$. 

Using Ramsey's Theorem on $N_H(v_s)$, we get either a clique of size $k$ or an independent set of size $k$. 
If there is an independent set $\{v_{i_1}, \ldots, v_{i_k}\}$ in $N_H(v_s)$ where $i_1<i_2<\cdots<i_k$, then the graph induced on $\{v_s\}\cup\{v_{i_z}, x_{i_z}: z\in\{1, \ldots, k\}\}\cup \{y_{z}:z\in\{i_1, i_1+1, \ldots, i_k\}\}$ is a subdivision of $F_k$. Thus, it contains a vertex-minor isomorphic to $F_k$. 
If there is a clique $\{v_{i_1}, \ldots, v_{i_k}\}$ in $N_H(v_s)$ where $i_1<i_2<\cdots<i_k$, then first apply local complementation at $v_s$ to change $\{v_{i_1}, \ldots, v_{i_k}\}$ into an independent set. 
Similar to above, the graph induced on $\{v_s\}\cup\{v_{i_z}, x_{i_z}: z\in\{1, \ldots, k\}\}\cup \{y_{z}:z\in\{i_1, i_1+1, \ldots, i_k\}\}$ is a subdivision of $F_k$, which contains a vertex-minor isomorphic to $F_k$.
\end{proof}

Now, it is sufficient to find a vertex-minor isomorphic to a graph described in Proposition~\ref{prop:ektofan}.
In Subsection~\ref{sec:patchedpath}, we show how to extract an induced matching between two levels in a leveling where one contains a long induced path.

\subsection{$\ell$-patched paths}\label{sec:patchedpath}

The following proposition will be used to extract an induced matching between two levels in a leveling where one level contains a long induced path.

\begin{proposition}\label{prop:indepandpath}
Let $k\ge 3$ and $\ell\ge 1$ be integers. 
Let $G$ be a graph on the disjoint union of vertex sets $S$ and $T$ 
such that $G[T]$ is an induced path and
each vertex of $T$ has a neighbor in $S$.
If $\abs{T}\ge (k-1)^{(k-1)^{2\ell+1}+1}$, 
then either $S$ has a vertex having at least $k$ neighbors in $T$, or
there exist $S'\subseteq S$, $T'\subseteq T$ with $S'=\{s'_j:1\le j\le \ell\}$, $T'=\{q'_j:1\le j\le \ell\}$ and a graph $G'$ on the vertex set $S'\cup T'$ such that 
\begin{itemize}
\item $G'[S']=G[S']$ and $G'[T']$ is an induced path $q_1'q_2' \cdots q_{\ell}'$,
\item $s'_i$ is adjacent to $q'_j$ in $G'$ if and only if $i=j$, and
\item $G'$ is obtained from $G$ by applying a sequence of local complementations at vertices in $T$ and removing vertices in $V(G)\setminus (S'\cup T')$.
\end{itemize}
\end{proposition}

For $\ell\ge 1$, an \emph{$\ell$-patched path} is a graph $G$ on two disjoint sets $S=\{s_1, s_2, \ldots, s_\ell\}$ and $T=\{q_1, q_2, \ldots, q_n\}$ 
satisfying the following.
\begin{itemize}
\item $G[T]$ is an induced path $q_1q_2 \cdots q_n$, called its \emph{underlying path}.
\item There exists 
a sequence $b_1< \ldots<b_2<\cdots< b_{\ell}\le n$ 
such that for each $j\in \{1,2,\ldots, \ell\}$, 
$s_j$ is adjacent to $q_{b_j}$
and non-adjacent to $q_{m}$ for all $m>b_j$.
\end{itemize}

In particular, if  $s_j$ has no neighbors in $\{q_1,\ldots,q_{b_{j-1}}\}$ for all $j\in\{2,\ldots,k\}$, then we call it a \emph{simple $\ell$-patched path}.

We first find an $\ell$-patched path with sufficiently large $\ell$ from the structure given in Proposition~\ref{prop:indepandpath}.
In the next step, we will find a long simple patched path from a patched path.

\begin{lemma}\label{lem:inducedmatching1}
Let $k\ge 3$ and $\ell\ge 1$ be  integers. Let $G$ be a graph on the disjoint union of vertex sets $S$ and $T$ 
such that $G[T]$ is an induced path and
each vertex of $T$ has a neighbor in $S$.
If $\abs{T}\ge 1+(k-1)+(k-1)^2+\cdots+(k-1)^{\ell}$, then 
either $S$ has a vertex having at least $k$ neighbors in $T$, or 
there exist $S'\subseteq S$ and $T'\subseteq T$ such that 
$G[S'\cup T']$ is an $\ell$-patched path whose underlying path is $G[T']$.
\end{lemma}
\begin{proof}
Suppose that every vertex of $S$ has less than $k$ neighbors in $T$.
Let $q_1q_2\ldots q_{\abs{T}}$ be the path induced by $T$.
Assume that $\abs{T}\ge 1+(k-1)+(k-1)^2+\cdots+(k-1)^\ell$.

Let  $s_1\in S$  be a neighbor of $q_1$. Since $s_1$ has at most $k-1$ neighbors on $T$, there exists $ b_1$ such that $q_{b_1}$ is adjacent to $s_1$ and 
$q_{b_1+j}$ is non-adjacent to $s_1$ for all \[1\le j\le 
\left\lceil \frac{1+(k-1)+(k-1)^2+\cdots+(k-1)^\ell}
{(k-1)}-1\right\rceil
= 1+(k-1)+(k-1)^2+\cdots+(k-1)^{\ell-1} \]
and $b_1\le (k-1)^\ell$.

Let $i$ be the maximum $i$ such that there exist distinct vertices $s_1,s_2,\ldots,s_i$ of $S$ 
and a sequence $b_1<b_2<\cdots<b_i$ such that 
\begin{itemize}
\item $b_1\le (k-1)^\ell$,  and $b_{m+1}-b_{m}\le (k-1)^{\ell-m}$  for all $1\le m<i$, 
\item  for all $1\le m\le i$, 
$s_m$ is adjacent to $q_{b_m}$ but non-adjacent to $q_{b_m+j}$ for all $1\le j\le 1+(k-1)+(k-1)^2+\cdots+(k-1)^{\ell-m}$.
\end{itemize}

Such $i$ exists, because $i=1$ satisfies the conditions.

Suppose that $i<\ell$. 
Let $s_{i+1}\in S$ be a neighbor of $q_{b_i+1}$.
For each $m\le i$, since $b_i+1-b_m\le (k-1)^{\ell-m}+(k-1)^{\ell-(m+1)}+\cdots+(k-1)^{\ell-(i-1)}+1
\le 1+(k-1)+(k-1)^2+\cdots+(k-1)^{\ell-m}$,
$s_m$ is non-adjacent to $q_{b_i+1}$ and therefore $s_m\neq s_{i+1}$.

Since $s_{i+1}$ has at most $k-1$ neighbors in $\{q_{b_{i}+j}: 1\le j\le 1+(k-1)+(k-1)^2+\cdots+(k-1)^{\ell-i}\}$, there exists $b_{i+1}$ such that 
$b_i+1\le b_{i+1}\le b_i+(k-1)^{\ell-i}$
and
$s_{i+1}$ is adjacent to $q_{b_{i+1}}$
but non-adjacent to $b_{i+1}+j$ for all \[1\le j\le 
\left\lceil \frac{1+(k-1)+\cdots+(k-1)^{\ell-i}}{k-1}-1\right\rceil =1+(k-1)+\cdots+(k-1)^{\ell-i-1}.\]
This contradicts our assumption that $i$ was maximum. 

Thus $i\ge \ell$. We take $S'=\{s_1,s_2,\ldots,s_\ell\}$
and $T'=\{q_1,q_2,\ldots,q_{b_\ell}\}$.
For all $m<\ell$, since $b_\ell-b_m=(k-1)^{\ell-m}+(k-1)^{\ell-(m+1)}+\cdots+(k-1)^1+1$, $s_m$ is non-adjacent to all $q_i$ with $b_m<i\le b_\ell$.
\end{proof}

\begin{lemma}\label{lem:inducedmatching2}
Let $k\ge 3$ and $\ell\ge 1$ be integers. If $G$ is a graph on the disjoint union of vertex sets $S$ and $T$ such that
$G$ is a $(1+(k-1)+(k-1)^2+\cdots+(k-1)^{\ell-1} )$-patched path whose underlying path is $G[T]$, 
then either $S$ has a vertex having at least $k$ neighbors in $T$, or 
there exist $S'\subseteq S$, $T'\subseteq T$ such that
$G[S'\cup T']$ is a simple $\ell$-patched path whose underlying path is $G[T']$.
\end{lemma}
\begin{proof}
  Suppose that every vertex of $S$ has at most $k-1$ neighbors in $T$.
  Suppose that $S=\{s_1, s_2, \ldots, s_{\abs{S}}\}$ and $G[T]$
  is an underlying induced path $q_1q_2\cdots q_m$. Furthermore let us assume that there exists a sequence $b_1<b_2<\cdots<b_{(k-1)^\ell}\le m$ such that 
  for all $i$, $s_i$ is adjacent to $q_{b_i}$ but non-adjacent to $q_j$ for all $j>b_i$.

  We prove a stronger claim that $T'$ can be chosen so that $T'=\{q_{i},q_{i+1},q_{i+2},\ldots,q_m\}$ for some $i$.
  We proceed by induction on $\ell$.
  The statement is trivial if $\ell=1$ and so we may assume $\ell>1$.

  We say that a vertex $q_j$ of $T$ is \emph{paired} with $s_i$ if $b_i=j$.
  There are $\abs{S}=1+(k-1)+(k-1)^2+\cdots+(k-1)^{\ell-1}$ paired vertices in $T$.
  We say that a paired vertex $q_j$ is an \emph{$s$-friend} of $q_t$ for $s\in S$
  if $j<t$ and $q_j,q_{j+1},\ldots,q_{t-1}$ are non-neighbors of $s$ and $q_t$ is a neighbor of $s$.

  Let $s'=s_{\abs{S}}$. Since $s'$ has at most $k-1$ neighbors in $T$, there exists $b'$ such that $s'$ is adjacent to $q_{b'}$ and 
  the number of $s'$-friends of $q_{b'}$ is at least 
  \[
  \left\lceil \frac{(1+(k-1)+(k-1)^2+\cdots+(k-1)^{\ell-1})-(k-1)}{k-1}\right\rceil
  = 1+(k-1)+\cdots+(k-1)^{\ell-2}.
  \]
  Let $S_1$ be a set of all $s_i\in S$ such that $q_{b_i}$ is an $s'$-friend of $q_{b'}$
  and $\abs{S_1}=1+(k-1)+\cdots+(k-1)^{\ell-2}$.
  Let $i$ be the minimum such that $q_i$ is paired with some $s\in S_1$.
  Let $T_1=\{q_i,q_{i+1},\ldots,q_{b'-1}\}$. 
  Then $G[S_1\cup T_1]$ is a $(1+(k-1)+\cdots+(k-1)^{\ell-2})$-patched path and therefore by the induction hypothesis, there exist $S'_1\subseteq S_1$, $T'_1\subseteq T_1$ such that $G[S'_1\cup T'_1]$  is a simple $(\ell-1)$-patched path whose underlying path is $G[T'_1]$
and furthermore $T'_1=\{q_p,q_{p+1},\ldots,q_{b'-1}\}$ for some $p$.

By the definition of an $s'$-friend, no vertex in $T_1$ is adjacent to $s'$.
Let $S'=S'_1\cup \{s'\}$ and $T'=T_1\cup \{q_{b'},q_{b'+1},\ldots,q_m\}$. Then $G[S'\cup T']$ is a simple $\ell$-patched path whose underlying path is $G[T']$.
\end{proof}

\begin{lemma}\label{lem:inducedmatching3}
Let $\ell$ be a positive integer. If $G$ is a graph on the disjoint union of vertex sets $S$ and $T$ such that
$G$ is a simple $2\ell$-patched path whose underlying path is $G[T]$,
then 
there exist $S'\subseteq S$, $T'\subseteq T$ with $S'=\{s'_j:1\le j\le \ell\}$, $T'=\{q'_j:1\le j\le \ell\}$ and a graph $G'$ on the vertex set $S'\cup T'$ such that 
\begin{itemize}
\item $G'[S']=G[S']$ and $G'[T']$ is an induced path $q_1'q_2' \cdots q_{\ell}'$,
\item $s'_i$ is adjacent to $q'_j$ in $G'$ if and only if $i=j$, and 
\item $G'$ is obtained from $G$ by applying a sequence of local complementations at vertices in $T$ and removing vertices in $V(G)\setminus (S'\cup T')$.
\end{itemize}
\end{lemma}
\begin{proof}
 Suppose that $S=\{s_1, s_2, \ldots, s_{2\ell}\}$ and $G[T]$
  is an underlying induced path $q_1q_2\cdots q_m$. Furthermore let us assume that there exists a sequence $0=b_0<b_1<b_2<\cdots<b_{2\ell}\le m$ such that 
  for all $i$, $s_i$ is adjacent to $q_{b_i}$ but non-adjacent to $q_j$ for all $j>b_i$ and all $j\le b_{i-1}$.
We proceed by induction on $\abs{V(T)}$.
  The statement is trivial if $\abs{V(T)}=2\ell$.  We assume that $\abs{V(T)}>2\ell$.

If $T$ contains a vertex of degree $2$ in $G$, then we smooth it. 
Since the resulting graph is still a simple  $2\ell$-patched path, we are done by induction hypothesis.

If $s_{i}$ is adjacent to $4$ consecutive neighbors $q_{x+1}, q_{x+2}, q_{x+3}, q_{x+4}$, 
then we apply local complementation at $ q_{x+2}$ and remove it.
This operation removes the edges $s_{i}q_{x+1}$ and $s_{i}q_{x+3}$.
Since $s_{i}$ has at least one neighbor $q_{x+4}$, 
the resulting graph is a simple $2\ell$-patched path, and  it contains the required structure by induction hypothesis.

By these two reductions, we may assume that 
each vertex in $T$ has a neighbor in $S$, and each vertex in $S$ has at most $3$ neighbors in $T$.

Now, we take a subset $S'=\{s_2, s_4, \ldots, s_{2\ell}\}$ of $S$, and let $G':=G[T\cup S']$.
For each $1\le i\le \ell$, 
we shrink the path $q_{b_{2(i-1)}+1}q_{b_{2(i-1)}+2} \cdots q_{b_{2i}}$ into some vertex $q'_i$
such that $q'_i$ is adjacent to $s_{2i}$.

If $\abs{N_G(s_{2i})\cap T}=1$, then let $q'_i:=q_{b_{2i}}$.
If $\abs{N_G(s_{2i})\cap T}=2$, then we apply local complementation at $q_{b_{2i}-1}$ and remove it. 
Then $s_{2i}q_{b_{2i}}$ is removed and $s_{2i}q_{b_{2i}-2}$ is added.
We assign $q'_i:=q_{b_{2i}-2}$.
In case when $\abs{N_G(s_{2i})\cap T}=3$, we pivot $q_{b_{2i}-2}q_{b_{2i}-1}$ and remove both end vertices. 
Then $s_{2i}q_{b_{2i}}$ is removed and $s_{2i}q_{b_{2i}-3}$ is added.
We assign $q'_i:=q_{b_{2i}-3}$.
We can observe that in each case, $s_{2i}$ has exactly one neighbor on the remaining path from $q_{b_{2(i-1)}+1}$ to $q_{b_{2i}}$.
Finally, we smooth all vertices of $q_{b_{2(i-1)}+1}, \ldots, q_{b_{2i}}$ except $q'_i$ in the remaining path.
Then we obtain an induced path $q'_1q'_2 \cdots q'_{\ell}$ such that
$s_{2i}$ is adjacent to $q'_j$ if and only if $i=j$.
\end{proof}

\begin{proof}[Proof of Proposition~\ref{prop:indepandpath}]
Suppose that every vertex of $S$ has at most $k-1$ neighbors in $T$.
Since $\abs{T}\ge (k-1)^{(k-1)^{2\ell+1}+1}$ and $k\ge 3$, by Lemma~\ref{lem:inducedmatching1}, 
there exist $S_1\subseteq S$ and $T_1\subseteq T$ such that 
$G[S_1\cup T_1]$ is an $(k-1)^{2\ell+1}$-patched path whose underlying path is $G[T_1]$.
Then, 
by Lemma~\ref{lem:inducedmatching2}, 
there exist $S_2\subseteq S_1$, $T_2\subseteq T_1$ such that
$G[S_2\cup T_2]$ is a simple $2\ell$-patched path whose underlying path is $G[T_2]$.
Lastly, by Lemma~\ref{lem:inducedmatching3}, 
there exist $S_3\subseteq S_2$, $T_3\subseteq T_2$ with $S_3=\{s'_j:1\le j\le \ell\}$, $T_3=\{q'_j:1\le j\le \ell\}$ and a graph $G'$ on the vertex set $S_3\cup T_3$ such that 
\begin{itemize}
\item $G'[S_3]=G[S_3]$ and $G'[T_3]$ is an induced path $q_1'q_2' \cdots q_{\ell}'$,
\item $s'_i$ is adjacent to $q'_j$ in $G'$ if and only if $i=j$, and 
\item $G'$ is obtained from $G$ by applying a sequence of local complementations at vertices in $T_2\subseteq T$ and removing vertices in $V(G)\setminus (S_3\cup T_3)$.\qedhere
\end{itemize}
\end{proof}

\subsection{Proof of Theorem~\ref{thm:mainfanvertexminor}}\label{sec:fanfreegraphs}

\begin{proof}[Proof of Theorem~\ref{thm:mainfanvertexminor}]

Let $q$ and $k$ be positive integers. 
If $k=1$, then it is trivial. Since $F_2$ is isomorphic to $C_3$, graphs having no vertex-minor isomorphic to $F_2$ are exactly forests, and we can color such graphs with $2$ colors.
Therefore, we may assume that $k\ge 3$.
Let $\ell:=R(k,k)^{2(k-1)^2-1}+1$ and $m:=(k-1)^{(k-1)^{2R(k+1, k\ell)+1}+1}$.
Let $G$ be a graph with maximum clique size  $q$ such that it has no vertex-minor isomorphic to $F_k$.
We claim that $G$ can be colored with $2(m-1)^{q-1}$ colors.

We may assume that $G$ is connected as we can color each connected component separately.
Let $v$ be a vertex of $G$ and for $i\ge 0$, let $L_i$ be the set of all vertices of $G$ whose distance to $v$ is $i$ in $G$. 
If each $L_j$ is $(m-1)^{q-1}$-colorable, then $G$ is $2(m-1)^{q-1}$-colorable. 
By Theorem~\ref{thm:gyarfas}, we may assume that there exists a level $L_n$ containing an induced path $P$ on $m$ vertices.

By Proposition~\ref{prop:ektofan}, it is sufficient to find a vertex-minor that is isomorphic to a graph obtained from the disjoint union of $E_{\ell}$ with the leaves $w_1, \ldots, w_{\ell}$ and a connected graph $H$ on at least $\ell$ vertices with pairwise distinct vertices $v_1, \ldots, v_{\ell}$, by identifying $v_i$ and $w_i$ for all $1\le i\le \ell$. 
We construct this graph based on the path $P$ and the leveling $L_0, \ldots, L_n$.

Since $L_0, \ldots, L_n$ is a leveling, each vertex in $P$ has a neighbor in $L_{n-1}$. If $n=1$, then we directly obtain a vertex-minor isomorphic to $F_k$. We may assume that $n\ge 2$.
Since $m=(k-1)^{(k-1)^{2R(k+1, k\ell)+1}+1}$, by Proposition~\ref{prop:indepandpath},
there exist $S=\{s_j:1\le j\le R(k+1, k\ell)\}\subseteq L_{n-1}$, $T=\{q_j:1\le j\le R(k+1, k\ell)\}\subseteq V(P)$, and a graph $G'$ on the vertex set $L_0\cup \cdots \cup L_{n-2}\cup S\cup T$ such that
\begin{itemize}
\item $G$ and $G'$ are identical on the vertex set $L_0\cup \cdots \cup L_{n-2}\cup S$, 
\item $G'[T]$ is an induced path $q_1q_2 \cdots q_{R(k+1, k\ell)}$, 
\item $s_i$ is adjacent to $q_j$ in $G'$ if and only if $i=j$, and 
\item $G'$ is obtained from $G$ by applying a sequence of local complementations at vertices in $P$ and removing vertices in $V(G)\setminus V(G')$.
\end{itemize}
Since $\abs{S}=R\left(k+1, k\ell \right)$, by Ramsey's Theorem,
$G'[S]$ contains a clique of size $k+1$ or an independent set of size $k\ell$.
If $G'[S]$ has a clique $C$ of size $k+1$, 
then for a vertex $s_i\in C$ with minimum $i$, 
$G'*v$ contains an induced subgraph isomorphic to a subdivision of $F_k$ and so $G$ has a vertex-minor isomorphic to $F_k$.
Thus we may assume that 
$G'[S]$ contains an independent set $S'$ of size $k\ell$.

Now, if there is a vertex in $L_{n-2}$ that has $k$ neighbors on $S'$ in $G'$, 
then $G'$ contains an induced subgraph isomorphic to a subdivision of $F_k$.
Thus, we may assume that each vertex in $L_{n-2}$ has at most $k-1$ neighbors on $S'$ in $G'$. It implies that $n\ge 3$.
Since each vertex of $S'$ has a neighbor in $L_{n-2}$ and $k\ell\ge (k-1)\ell+1$, 
there exist $\{w_1, \ldots, w_{\ell}\}\subseteq L_{n-2}$ and $\{x_1, \ldots, x_{\ell}\}\subseteq S'$ where
$w_i$ is adjacent to $x_j$ in $G'$ if and only if $i=j$.
For each $1\le i\le \ell$, let $y_i$ be the neighbor of $x_i$ contained in $T$.

Let $G''$ be the graph obtained from  
\[G'[L_0\cup \cdots \cup L_{n-3}\cup \{w_{z}, x_z: z\in\{1, \ldots, \ell\}\}\cup T]\]
by repeatedly removing degree-$1$ vertices and smoothing degree-$2$ vertices in $T$ other than $y_1, \ldots, y_{\ell}$.
In the resulting graph, the vertices $y_1, \ldots, y_{\ell}$ remain among vertices of $T$.
Note that $G'[L_0\cup \cdots \cup L_{n-3}\cup \{w_{z}: z\in\{1, \ldots, \ell\}\}]$
is connected because there is a path from each vertex to the vertex in $L_0$.
Also, the graph obtained from $G''[\{w_{z}, x_z, y_z: z\in\{1, \ldots, \ell\}\}]$ 
by removing edges in $G''[\{w_{z}: z\in\{1, \ldots, \ell\}]$ is isomorphic to $E_{\ell}$.
Therefore, by Proposition~\ref{prop:ektofan}, it contains a vertex-minor isomorphic to $F_k$.
\end{proof}

\section{Coloring graphs without $C_k$ pivot-minors}
\label{sec:4}
In this section, we prove the second main result.

\begin{theorem}\label{thm:mainpivotminor}
For each integer $k\ge 3$, the class of graphs having no pivot-minor isomorphic to a cycle
of length $k$ is $\chi$-bounded.
\end{theorem}

\subsection{Obtaining $C_k$ pivot-minor from a large incomplete fan}\label{sec:findingfan}

We show that for every fixed $k$, there exists $\ell$ with the same parity as $k$ such that every graph consisting of an induced path $P$ of length $\ell$ and a vertex $v$ not on $P$ where  $v$ is adjacent to the end vertices of $P$ contains a pivot-minor isomorphic to $C_k$.
This will support Theorem~\ref{thm:mainpivotminor}. 

\begin{proposition}\label{prop:pathtocycle}
Let $k\ge 3$ be an integer and  $n\geq 6 k^3-26 k^2+25 k-2$ such that $k\equiv n\pmod 2$.
If $G$ is a graph with a vertex $v$ such that $G\setminus v$ is an induced path $P$ of length $n$ and $v$ is adjacent to the end vertices of $P$, then 
$G$ contains a pivot-minor isomorphic to $C_k$.
\end{proposition}

We remark that the parity condition in Proposition~\ref{prop:pathtocycle} cannot be removed as $C_n$ has no pivot-minor isomorphic to $C_{k}$ if $n\not\equiv k\pmod 2$.

To prove Proposition~\ref{prop:pathtocycle}, we prove some useful lemmas. 

\begin{lemma}\label{lem:cycletocycle}
Every induced cycle of length $k+2$ contains an induced cycle of length $k$ as a pivot-minor.
\end{lemma}
\begin{proof}
By pivoting an edge $xy$ on an induced cycle and deleting $x, y$ from the resulting graph, we obtain an induced cycle that is of length $2$ shorter than the initial one. 
\end{proof}

\begin{lemma}\label{lem:evenandodd}
Let $G$ be a graph with a vertex $v$ such that $G\setminus v$ is an induced path $P:=p_0p_1 \cdots p_n$. 
Let $i_1=0<i_2<i_3<\cdots<i_t=n$ be a sequence of integers such that $p_{i_1}$, $\ldots$, $p_{i_t}$ are all neighbors of $v$ on $P$.
Then the following hold.
\begin{enumerate}[(1)]
\item If $k:=i_2-i_1>1$ and $i_2\equiv i_3\equiv \cdots \equiv i_{t-1}\not\equiv i_t \pmod 2$, then $G$ contains a pivot-minor isomorphic to $C_{k+1}$. 
\item For a positive integer $k$, if $t\ge 4k$ and $i_j=j-1$ for all $j\in \{1,2,\ldots,t\}$, then $G$ contains a pivot-minor isomorphic to $C_{2k+1}$ and a pivot-minor isomorphic to  $C_{2k+2}$.
\item For a positive integer $k$, if $t\ge 2k+1$ and $i_j=2(j-1)$ for all $j\in\{1, \ldots, t-1\}$, then $G$ contains a pivot-minor isomorphic to $C_{2k+2}$.
Moreover, if $i_{t}-i_{t-1}$ is odd, then $G$ contains a pivot-minor isomorphic to $C_{2k+1}$.
\end{enumerate}
\end{lemma}

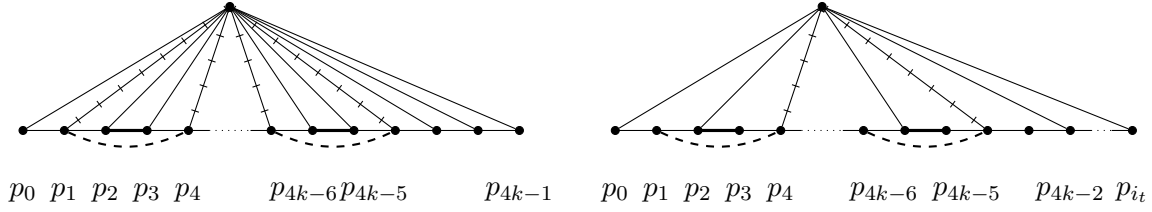
\begin{figure}[t]
  \centering
  \tikzstyle{v}=[circle, draw, solid, fill=black, inner sep=0pt, minimum width=3pt]
      \begin{tikzpicture}[scale=0.55]
    \foreach \x in {1,...,5}
    {
      \node (v\x) at (\x, 1) [v] {};
    }
    \foreach \x in {7,...,13}
    {
      \node (v\x) at (\x, 1) [v] {};
    }
    \draw (1,1)--(5.5,1);
    \draw (6.5,1)--(13,1);
    \draw[dotted] (5.5,1)--(6.5,1);
   \draw[very thick] (3,1)--(4,1);
   \draw[very thick] (8,1)--(9,1);
      \node (w) at (6, 4) [v] {};

\path (v2) edge [thick, bend left=-25, dashed]  (v5);   
\path (v7) edge [thick, bend left=-25, dashed]  (v10);   

\draw (1, -.5) node{$p_0$};
\draw (2, -.5) node{$p_1$};
\draw (3, -.5) node{$p_2$};
\draw (4, -.5) node{$p_3$};
\draw (5, -.5) node{$p_4$};

\draw (7.8, -.5) node{$p_{4k-6}$};
\draw (9.5, -.5) node{$p_{4k-5}$};
\draw (13, -.5) node{$p_{4k-1}$};

   \foreach \x in {1,3,4,8,9,11,12,13}
    {
\draw (w)--(v\x);
}
   \foreach \x in {2, 5, 7, 10}
    {
\draw (w)--(v\x);
\draw[snake=ticks, segment amplitude=1.5pt] (w) -- (v\x);
}

  \end{tikzpicture}\quad
        \begin{tikzpicture}[scale=0.55]
    \foreach \x in {1,...,5}
    {
      \node (v\x) at (\x, 1) [v] {};
    }
    \foreach \x in {7,...,12}
    {
      \node (v\x) at (\x, 1) [v] {};
    }
      \node (v13) at (13.5, 1) [v] {};
    \draw (1,1)--(5.5,1);
    \draw (6.5,1)--(12.5,1);
    \draw[dotted](12.5,1)--(13,1);
    \draw (13, 1)--(13.5,1);
    \draw[dotted] (5.5,1)--(6.5,1);
   \draw[very thick] (3,1)--(4,1);
   \draw[very thick] (8,1)--(9,1);
      \node (w) at (6, 4) [v] {};

\path (v2) edge [thick, bend left=-25, dashed]  (v5);   
\path (v7) edge [thick, bend left=-25, dashed]  (v10);   

\draw (1, -.5) node{$p_0$};
\draw (2, -.5) node{$p_1$};
\draw (3, -.5) node{$p_2$};
\draw (4, -.5) node{$p_3$};
\draw (5, -.5) node{$p_4$};

\draw (7.5, -.5) node{$p_{4k-6}$};
\draw (9.5, -.5) node{$p_{4k-5}$};
\draw (12, -.5) node{$p_{4k-2}$};
\draw (13.5, -.5) node{$p_{i_t}$};

   \foreach \x in {1,3,8,12,13}
    {
\draw (w)--(v\x);
}
   \foreach \x in {5, 10}
    {
\draw (w)--(v\x);
\draw[snake=ticks, segment amplitude=1.5pt] (w) -- (v\x);
}

  \end{tikzpicture}
  \caption{Configurations in (2) and (3) of Lemma~\ref{lem:evenandodd}.}
  \label{fig-prop-odd-even}
\end{figure}

\begin{proof}
(1) We proceed by induction on $i_t-i_2$. 

If $i_t-i_2=1$, then we can create the edge $vp_{i_2-1}$ by pivoting the edge $p_{i_2}p_{i_t}$. Since $p_{i_2},p_{i_t}$ have no neighbors in $\{p_{i_1}, \ldots, p_{i_2-2}\}$, $vp_{0}p_{1} \cdots p_{i_2-1}v$ is an induced cycle of length $k+1$ in $G\wedge p_{i_2}p_{i_t}$. 

If $i_t-i_2\ge 3$, then we can create the edge $vp_{i_t-2}$ by pivoting $p_{i_{t-1}}p_{i_t}$.
Then the (new) neighborhood of $v$ on the path from $p_{i_1}$ to $p_{i_t-2}$ satisfies the condition of our assumption as the new edge $vp_{i_t-2}$ divides either an even interval into two odd intervals or an odd interval into an odd interval and an interval of length $2$.
Thus, by the induction hypothesis, $G\wedge p_{i_{t-1}}p_{i_t}$ contains a pivot-minor isomorphic to $C_{k+1}$ and so does $G$.

\medskip
(2) For $j\in \{1, \ldots, t-3\}$, if we pivot $p_{j}p_{j+1}$, then the edges $vp_{j-1}$, $vp_{j+2}$ are removed and $p_{j-1}p_{j+2}$ is added.
If $k\ge 2$, then by pivoting $p_{4j-2}p_{4j-1}$ and removing the vertices $p_{4j-2}$ and $p_{4j-1}$ for all $j\in\{1, \ldots, k-1\}$, we can obtain an induced cycle 
\[vp_{0}p_{1}p_{4}p_{5} \cdots p_{4k-4}p_{4k-3}v\] 
of length $2k+1$. If $k=1$, then $vp_0p_1$ is an induced cycle of length $3=2k+1$.
Now, by pivoting $p_{4k-2}p_{4k-1}$, we can remove the edge $vp_{4k-3}$ and thus, we obtain an induced cycle of length $2k+2$, which is $vp_{0}p_{1} \cdots p_{4k-4}p_{4k-3}p_{4k-2}v$.

\medskip
(3) For $j\in\{1, \ldots, t-3\}$, if we pivot $p_{2j}p_{2j+1}$, then the edge $vp_{2j+2}$ is removed and $p_{2j-1}p_{2j+2}$ is added.
Therefore, pivoting $p_2p_3, p_6p_7,p_{10}p_{11},\ldots,p_{4k-6}p_{4k-5}$ and removing the vertices $p_2$, $p_3$, $p_6$, $p_7$, $p_{10}$, $p_{11}$, $\ldots$, $p_{4k-6},p_{4k-5}$ creates an induced cycle
 \[vp_{0}p_{1}p_{4}p_{5} \cdots p_{4k-6}p_{4k-5}p_{4k-4}v\] of length $2k+2$.
If $i_t-i_{t-1}$ is odd, then the last odd interval is still an odd interval after pivotings, and by (1), it also contains a pivot-minor isomorphic to a cycle of length $2k+1$.
\end{proof}

For positive integers $k, \ell$, a \emph{$(k,\ell)$-fan} is a graph $F$ with a specified vertex $p$, called the \emph{central vertex},  such that 
\begin{itemize}
\item $F\setminus p$ is a path $p_0p_1 \cdots p_{n}$, and 
let $i_1=0<i_2<i_3<\cdots<i_t=n$ be a sequence of integers such that $p_{i_1}$, $\ldots$, $p_{i_t}$ are all neighbors of $v$ on $P$,
\item $i_{j+1}-i_j$ is odd for $j\in\{1, \ldots, k\}$,
\item $\abs{ j\in\{1, \ldots, t-1\}:i_{j+1}-i_j\text{ is odd}\}       }\ge \ell$.
\end{itemize}

\begin{lemma}\label{lem:evenandodd2}
  Every $(k,\ell)$-fan   contains a pivot-minor isomorphic to $F_{k+\lfloor (\ell-k)/3 \rfloor}$.
\end{lemma}
\begin{proof}
  Let $m=k+\lfloor (\ell-k)/3 \rfloor$.
  Let $G$ be the $(k,\ell)$-fan with the central vertex $v$ such that $G\setminus v$ is an induced path $P:=p_0p_1 \cdots p_n$ and let $i_1=0<i_2<i_3<\cdots<i_t=n$ be a sequence of integers such that $p_{i_1}$, $\ldots$, $p_{i_t}$ are all neighbors of $v$ on $P$. 
 
  We proceed by induction on $\abs{V(G)}-k$.
  If there exists $j$ such that both $p_j$ and $p_{j+1}$ are non-adjacent to $v$, 
  then $G\wedge p_jp_{j+1}\setminus p_jp_{j+1}$ is a $(k,\ell)$-fan, thus having a pivot-minor isomorphic to $F_m$ by the induction hypothesis. Thus we may assume that $i_{j+1}-i_j\in\{1,2\}$ for all $j\in \{1,2,\ldots,t-1\}$. If $\ell-k<3$, then $G$ contains an induced subgraph isomorphic to $F_m$. Thus we may assume that $\ell-k\ge 3$. 
  
  If $i_{k+2}-i_{k+1}$ is odd, then $G$ is a $(k+1,\ell)$-fan and therefore by the induction hypothesis, $F_m$ is isomorphic to a pivot-minor of $G$. Thus we may assume that $i_{k+2}-i_{k+1}=2$ and therefore $i_j=j-1$ for all $j\in\{1,2,\ldots,k+1\}$ and $i_{k+2}=k+2$. 

  If $p_{k+3}$ is non-adjacent to $v$, then $p_{k+4}$ is adjacent to $v$ and $G\wedge p_{k+2}p_{k+3}\setminus p_{k+2}p_{k+3}$ is a  $(k+1,\ell)$-fan, proving this lemma by the induction hypothesis. Thus we may assume that $p_{k+3}$ is adjacent to $v$ and $i_{k+3}=k+3$.

  If $p_{k+4}$ is non-adjacent to $v$, then $G\wedge p_{k+2}p_{k+3}\setminus p_{k+2}p_{k+3}$ is a $(k+1,\ell)$-fan. Thus, we may assume that $p_{k+4}$ is adjacent to $v$ and $i_{k+4}=k+4$.

  Now, $G\wedge p_{k+2}p_{k+3}\setminus p_{k+2}\setminus p_{k+3}$ is a $(k+1,\ell-3)$-fan, thus having a pivot-minor isomorphic to $F_m$ by the induction hypothesis.
\end{proof}

Now we are ready to prove Proposition~\ref{prop:pathtocycle}.

\begin{proof}[Proof of Proposition~\ref{prop:pathtocycle}]
Let $P:=p_0p_1 \cdots p_n$ and let $i_1=0<i_2<i_3<\cdots<i_t=n$ be a sequence of integers such that $p_{i_1}$, $\ldots$, $p_{i_t}$ are all neighbors of $v$ on $P$.

If $i_{j+1}-i_j\ge k-2$ and $i_{j+1}-i_j\equiv k\pmod 2$ for some $j$, then $G$ has a pivot-minor isomorphic to $C_k$ by Lemma~\ref{lem:cycletocycle}.

If $i_{j+1}-i_j\ge k-2$ and $i_{j+1}-i_j\not\equiv k\pmod 2$ for some $j$, then there exists $m$ such that $i_{m+1}-i_m$ is odd, because $n\equiv k\pmod 2$. By symmetry, we may assume that $m>j$. We may assume that $m$ is chosen to be minimum.
Then, $i_{j+1}\equiv i_{j+2}\equiv \cdots \equiv i_m\not\equiv i_{m+1} \pmod 2$ and therefore $G$ contains a pivot-minor isomorphic to $C_k$ by (1) of Lemma~\ref{lem:evenandodd}.
Thus we may assume that $i_{j+1}-i_j\le k-3$ for all $j$
and therefore $n\le (k-3)(t-1)$.

If there exist at least $6k-2$ values of $j$ such that $i_{j+1}-i_j$ is odd, then $G$ has a $(1,6k-2)$-fan as an induced subgraph and therefore by Lemma~\ref{lem:evenandodd2}, $G$ has a pivot-minor isomorphic to $F_{2k}$.
By (2) of Lemma~\ref{lem:evenandodd}, if $k$ is even, then $F_{2k}$ contains a pivot-minor isomorphic to $C_{k+2}$.
If $k$ is odd, then $F_{2(k-1)}$ contains a pivot-minor isomorphic to $C_k$ by (2) of Lemma~\ref{lem:evenandodd}.
Therefore we may assume that there are at most $6k-3$ values of $j$ such that $i_{j+1}\not\equiv i_j\pmod 2$.

Suppose that $i_j\equiv i_{j+1}\equiv i_{j+2}\equiv \cdots\equiv i_{j+k-1}\pmod 2$ for some $j\le t-k+1$. If $k$ is even, then by (3) of Lemma~\ref{lem:evenandodd}, $G$ has a pivot-minor isomorphic to $C_k$. If $k$ is odd, then there exists $m$ such that $i_{m+1}-i_m$ is odd. By (3) of Lemma~\ref{lem:evenandodd}, $G$ has a pivot-minor isomorphic to $C_k$.
Thus we may assume that at least one of $i_{j+1}-i_j, i_{j+2}-i_{j+1},\ldots,i_{j+k-1}-i_j$ is odd for all $j\le t-k+1$.
We conclude that $t\le (k-1)(6k-2)$
and therefore $n\le (k-3)((k-1)(6k-2)-1)=6 k^3-26 k^2+25 k-3$.
\end{proof}

\subsection{Proof of Theorem~\ref{thm:mainpivotminor}}\label{sec:nockpmchibound}

\begin{proof}[Proof of Theorem~\ref{thm:mainpivotminor}]
Let $q$ and $k$ be positive integers with $k\ge 3$.
If $k=3$, then graphs having no pivot-minor isomorphic to $C_3$ are bipartite graphs, and we can color such graphs with $2$ colors.
We may assume that $k\ge 4$.
Let $\ell:=6 k^3-26 k^2+25 k-2$.
Let $G$ be a graph such that it has no pivot-minor isomorphic to $C_k$.
We claim that $\chi(G)\le 2(\ell+1)^{q-1}$ if $\omega(G)\le q$.

We may assume that $G$ is connected as we can color each connected component separately.
Let $v$ be a vertex of $G$ and for $i\ge 0$, let $L_i$ be the set of all vertices of $G$ that are at distance $i$ away from $v$. 
If each $L_j$ is $(\ell+1)^{q-1}$-colorable, then $G$ is $2(\ell+1)^{q-1}$-colorable. 
By Theorem~\ref{thm:gyarfas}, we may assume that there exists a level $L_n$ 
containing an induced path of length $t\in \{\ell, \ell+1\}$ where $t$ and $k$ have the same parity.
Let $P:=p_0p_1p_2 \cdots p_t$. 
If $n=1$, then by Proposition~\ref{prop:pathtocycle}, 
$G[V(P)\cup \{v\}]$ contains a pivot-minor isomorphic to $C_k$.
We may assume that $n\ge 2$.

\begin{figure}
  \centering
  \tikzstyle{v}=[circle, draw, solid, fill=black, inner sep=0pt, minimum width=3pt]
  \begin{tikzpicture}
    \foreach \x in {0,...,6}
    {
      \node (v\x) at (\x,0) [v] {};
    }
    \node at (0,0) [v,label=below:$p_0$] {};
    \node at (1,0) [v,label=below:$p_1$] {};
    \node at (2,0) [v,label=below:$p_2$] {};
    \node at (4,0) [v] {};
    \node at (5,0) [v,label=below:$p_{t-1}$] {};
    \node at (6,0) [v,label=below:$p_t$] {};

      \node (x) at (1,1) [v,label=left:$x$] {};
      \node (y) at (5,1) [v,label=right:$y$] {};

    \draw [thick](v0)--(v3);
    \draw [thick, dotted](v3)--(v4);
    \draw [thick](v4)--(v6);
    \draw[thick] (x)--(v0);\draw[thick](y)--(v6);

    \draw [dotted] (x)--(1,0.5);
    \draw [dotted] (x)--(1.5,0.5);
    \draw [dotted] (x)--(2,0.5);
    \draw [dotted] (y)--(5,0.5);
    \draw [dotted] (y)--(4.5,0.5);
    \draw [dotted] (y)--(4,0.5);
    \draw [dotted](x)--(y);
    
    \node (p1) at (2,2) [v] {};
    \node (p2) at (2,3) [v] {};
    \node (p3) at (2,4) [v,label=left:$z_1$] {};
    \node (q1) at (4,2) [v] {};
    \node (q2) at (4,3) [v] {};
    \node (q3) at (4,4) [v,label=right:$z_2$] {};
    \node (z) at (3,5) [v,label=above:$z$] {};
    
    \draw [thick] (x)--(p1)--(p3)--(z)--(q3)--(q1)--(y);
    \draw [dotted] (p1)--(q1);
    \draw [dotted] (p2)--(q2);
    \draw [dotted] (p3)--(q3);

 \node at (3,-1) {$G_2$};
  \end{tikzpicture}\quad
   \begin{tikzpicture}
    \foreach \x in {0,...,6}
    {
      \node (v\x) at (\x,0) [v] {};
    }
    \node at (0,0) [v,label=below:$p_0$] {};
    \node at (1,0) [v,label=below:$p_1$] {};
    \node at (2,0) [v,label=below:$p_2$] {};
    \node at (4,0) [v] {};
    \node at (5,0) [v,label=below:$p_{t-1}$] {};
    \node at (6,0) [v,label=below:$p_t$] {};

      \node (x) at (1,1) [v,label=left:$x$] {};
      \node (y) at (5,1) [v,label=right:$y$] {};

    \draw [thick](v0)--(v3);
    \draw [thick, dotted](v3)--(v4);
    \draw [thick](v4)--(v6);
    \draw[thick] (x)--(v0);\draw[thick](y)--(v6);

    \draw [dotted] (x)--(1,0.5);d
    \draw [dotted] (x)--(1.5,0.5);
    \draw [dotted] (x)--(2,0.5);
    \draw [dotted] (y)--(5,0.5);
    \draw [dotted] (y)--(4.5,0.5);
    \draw [dotted] (y)--(4,0.5);
    \draw [dotted](x)--(y);
    
    \node (p1) at (2,2) [v] {};
    \node (p2) at (2,3) [v] {};

    \node (q1) at (4,2) [v] {};
    \node (q2) at (4,3) [v] {};
    \node (q3) at (4,4) [v,label=right:$z_2$] {};

    \draw [thick] (x)--(p1)--(p2)--(q3)--(q1)--(y);
    \draw [dotted] (p1)--(q1);
    \draw [dotted] (p2)--(q2);

    \draw [thick] (q3)--(p2);
 \node at (3,-1) {$G_2\wedge zz_1\setminus \{z, z_1\}$};

  \end{tikzpicture}
  \caption{Reducing the length of the path $x-P_1-z-P_2-y$ in Theorem~\ref{thm:mainpivotminor}.}
  \label{fig:reducelength}
\end{figure}
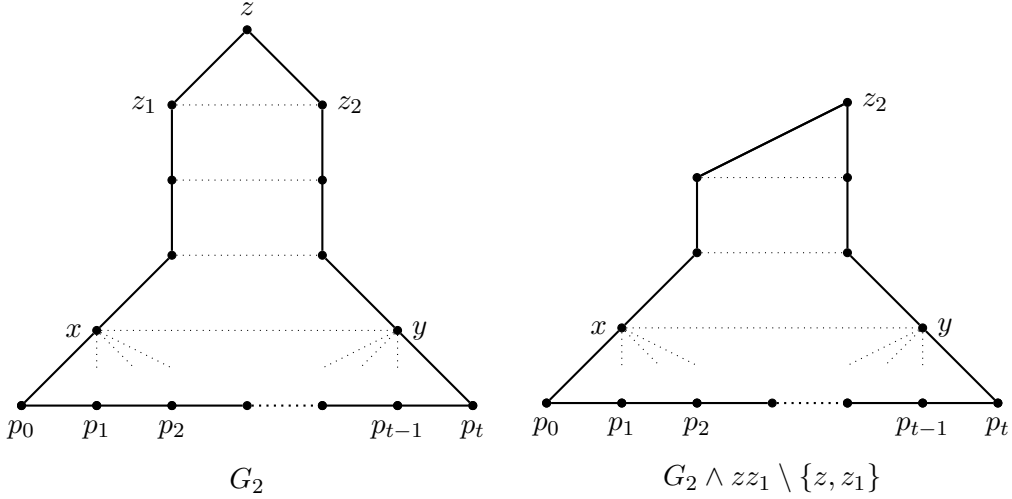

Let $x$ be a parent of $p_0$. If $x$ is adjacent to $p_t$, then by Proposition~\ref{prop:pathtocycle}, $G$ contains a pivot-minor isomorphic to $C_k$. 
We may assume that $x$ is not adjacent to $p_t$. Let $y$ be a parent of $p_t$. 
By the same reason, we can assume that $y$ is not adjacent to $p_0$.
We choose a first common ancestor of $x$ and $y$ in the leveling $L_0, \ldots, L_{n-1}$, and call it $z$.
Such a vertex $z$ exists because $v$ is a common ancestor of $x$ and $y$.
Let $P_1$ be the path from $x$ to $z$ in $G_1$ with exactly one vertex in each level, and similarly, let $P_2$ be the path from $y$ to $z$ in $G_1$ with exactly one vertex in each level. Since $P_1$ and $P_2$ have the same length, the path
$x-P_1-z-P_2-y$
has even length. Note that the path $x-P_1-z-P_2-y$ is not necessary an induced path in $G_1$ as there may be an edge between two vertices on the same level. See Figure~\ref{fig:reducelength}.

We claim that 
$G[V(P)\cup V(P_1)\cup V(P_2)]$ contains a pivot-minor isomorphic to $C_k$. 
Let $G_1:=G[V(P)\cup V(P_1)\cup V(P_2)]$.
Note that by construction, all internal vertices of the path $x-P_1-z-P_2-y$ have no neighbors on the path $P$. If there are at least two internal vertices in $x-P_1-z-P_2-y$, then let $z_1$ and $z_2$ be the neighbors of $z$ on $P_1$ and $P_2$, respectively.
We pivot $zz_1$ and remove $z$ and $z_1$ from $G_1$. 
Then $z_2$ becomes adjacent to the neighbor of $z_1$ on $P_1$ other than $z$.
This operation reduces the length of the path $x-P_1-z-P_2-y$ by $2$. Thus, we can do this until the remaining path has length exactly $2$. 
From this operation, we may assume that the path $x-P_1-z-P_2-y$ has length exactly $2$, which is $xzy$.

Now, we pivot $xz$ in $G_2$. Note that
\begin{itemize}
\item $p_0$ is adjacent to $x$ but not adjacent to $z$, and
\item $y$ is either a common neighbor of $x$ and $z$, or adjacent to $z$ but not to $x$.
\end{itemize}
From these two facts, $p_0y$ becomes an edge after pivoting $xz$. 
Since all vertices on $P$ are not adjacent to $z$, $V(P)$ still induces the same path after pivoting $xz$.
So, $y$ is adjacent to $p_0$ and $p_t$ in $G_2\wedge xz$, and by Proposition~\ref{prop:pathtocycle}, $G_2\wedge xz$ contains a pivot-minor isomorphic to $C_k$.
\end{proof}

\section{Further discussions}\label{sec:discuss}
Let us conclude our paper by summarizing known cases for Conjectures~\ref{con:geelen} and \ref{con:pivotminor}. 
As far as we know, 
the class of graphs having no $H$ vertex-minor is $\chi$-bounded if 
\begin{itemize}
\item $H$ is a distance-hereditary graph (due to Theorem~\ref{thm:scott}), 
\item $H$ is a vertex-minor of a fan graph (Theorem~\ref{thm:mainfanvertexminor}),
\item $H=W_5$ (due to Dvo\v{r}\'{a}k and Kr\'{a}l'~\cite{DvorakK2012}),
\end{itemize}
and 
the class of graphs having no $H$ pivot-minor is $\chi$-bounded if 
\begin{itemize}
\item $H$ is a pivot-minor of a cycle graph (Theorem~\ref{thm:mainpivotminor}),
\item $H$ is a pivot-minor of a $1$-subdivision of a tree, which we can deduce easily from Theorem~\ref{thm:scott},
\item $H$ is a pivot-minor of a tree satisfying Gy\'{a}rf\'{a}s-Sumner conjecture, which we describe below.
\end{itemize}
Gy\'arf\'as~\cite{Gyarfas1975} and Sumner~\cite{Sumner1981} independently conjectured that for a fixed tree $T$, the class of graphs having no induced subgraph isomorphic to $T$ is $\chi$-bounded. So far this conjecture is known to be true for the following cases:
\begin{itemize}
\item $T$ is a subdivision of a star (due to Scott~\cite{Scott1997}),
\item $T$ is a tree of radius $2$ (due to Kierstead and Penrice~\cite{KiersteadR1994}),
\item $T$ is a tree of radius $3$ obtained from a tree of radius $2$ by making exactly one subdivision in every edge adjacent to the root (due to Kierstead and Zhu~\cite{KiersteadZ2004}).
\end{itemize}
Note that a cycle is a vertex-minor of a large fan graph. Thus, Conjecture~\ref{con:geelen} holds when $H$ is a cycle graph, by two reasons, 
one  by Theroem~\ref{thm:mainfanvertexminor} 
and another by the proof of (ii) of Conjecture~\ref{con:gy} by Scott and Seymour~\cite{ScSe_1}.

One may wish to have a structure theorem describing graphs with no fixed vertex-minors or no fixed pivot-minors in order to extend these theorems to other forbidden graphs.
Indeed, Oum~\cite{Oum2006a} conjectured the following.
A graph is a \emph{circle graph} if it is an intersection graph of chords in a circle.
Rank-width is a width parameter of graphs introduced by Oum and Seymour~\cite{OS2004}. 
\begin{conjecture}\label{con:oum}
  Let $H$ be a bipartite circle graph. Every graph with sufficiently large rank-width contains a pivot-minor isomorphic to $H$. 
\end{conjecture}
This conjecture, if true, implies $\chi$-boundedness by the following theorem of 
Dvo\v{r}\'{a}k and Kr\'{a}l'~\cite{DvorakK2012}.
\begin{theorem}[Dvo\v{r}\'{a}k and Kr\'{a}l'~\cite{DvorakK2012}]\label{thm:kral}
  For each integer $k$, the class of graphs of rank-width at most $k$ is $\chi$-bounded.
\end{theorem}
Let $F_n'$ be a graph obtained from $F_n$ by subdividing each edge on the induced path precisely once. It can be easily seen that $F_n'$ is a bipartite circle graph and $F_n$ is a vertex-minor of $F_n'$. 
Thus if Conjecture~\ref{con:oum} holds, then the class of graphs with no $F_n$ vertex-minor has bounded rank-width and therefore by Theorem~\ref{thm:kral}, it will be $\chi$-bounded, implying Theorem~\ref{thm:mainfanvertexminor}.
Similarly we can also see easily that Conjecture~\ref{con:oum} implies Theorem~\ref{thm:mainpivotminor}.
However, we do not know yet whether Conjecture~\ref{con:oum} holds when $H=F_n'$ or $H$ is an even cycle.

Furthermore it would be interesting to see whether Conjectures~\ref{con:geelen} and \ref{con:pivotminor} hold when $H$ is a wheel graph on at least $6$ vertices, since such a graph $H$ is not a circle graph and therefore Conjectures~\ref{con:geelen} and \ref{con:pivotminor}  are independent of Conjecture~\ref{con:oum}.

\end{document}